\documentclass[12pt]{article}
\usepackage[english]{babel}
\usepackage{amsmath,amssymb,amsthm,xspace}

\RequirePackage{ifpdf}
\ifpdf
   \usepackage[pdftex]{hyperref}
\else
   \usepackage[hypertex]{hyperref}
\fi

\title{Infinite sharply multiply transitive groups}
\author{Katrin Tent\footnote{Partially supported by SFB 878}}
\date{\today}

\newtheorem{theorem}{Theorem}[section]

\newtheorem{lemma}[theorem]{Lemma}

\newtheorem{corollary}[theorem]{Corollary}
\newtheorem{definition}[theorem]{Definition}

\newcommand{\nc}{\newcommand}
\nc{\inv}{^{-1}}
\nc{\Q}{\mathbb{Q}}
\nc{\R}{\mathbb{R}}
\nc{\F}{\mathbb{F}}
\nc{\C}{\mathcal{C}}
\nc{\G}{\mathcal{G}}
\nc{\M}{\mathcal{M}}
\nc{\U}{\mathbb{U}}
\nc{\Frl}{Fra\"iss\'e limit\xspace}
\nc{\Frls}{Fra\"iss\'e limits\xspace}
\nc{\fg}{finitely generated\xspace}

\renewcommand{\phi}{\varphi}

\DeclareMathOperator{\AGL}{AGL}
\DeclareMathOperator{\PGL}{PGL}
\DeclareMathOperator{\Sym}{Sym}

\newcommand{\Ind}{
 \setbox0=\hbox{$x$}\kern\wd0\hbox to 0pt{\hss$
 \mid$\hss}\lower.9\ht0\hbox to 0pt{\hss$\smile$\hss}\kern\wd0
}

\begin{document}
\maketitle
\begin{abstract}
The finite sharply $2$-transitive groups were classified by Zassenhaus
in the 1930's. They essentially all look like the group
of affine linear transformations $x\mapsto ax+b$ for some field (or at least \emph{near-field}) $K$. However, the question remained open whether the same is true for infinite sharply $2$-transitive groups. There has been extensive work on the structures associated to such groups indicating that Zassenhaus' results might extend to the infinite setting. For many specific classes of groups, like Lie groups, linear groups, or groups definable in o-minimal structures it was indeed proved that all examples inside the given class arise in this way as affine groups. However, it recently turned out that the reason for the lack of a general proof was the fact that there are plenty of sharply $2$-transitive groups which do not arise from fields or near-fields! In fact, it is not too hard to construct concrete examples (see below). In this note, we survey general sharply $n$-transitive groups and describe how to construct examples not arising from fields.
\end{abstract}

\section{Historic background}

Groups often arise from symmetries of certain objects, like polyeders, geometries, manifolds etc. In these  examples, the groups automatically
come with an action on the underlying object and can thus be considered
as permutation groups. It is a natural question what restrictions occur in 
these groups and their corresponding actions and to characterize those
actions with a very high degree of symmetry.

One criterion to measure symmetry is the degree of \emph{transitivity}
of the group: recall that a group $G$ acting on a set $X$ is called
$n$-transitive if for any two $n$-tuples $(x_1,\ldots, x_n),(y_1,\ldots, , y_n)$
of \emph{distinct} elements from $X$ there is some $g\in G$ such that
$x_i^g=y_i, i=1,\ldots , n$. An $n$-transitive group action is \emph{sharply} $n$-transitive
if for any two such pairs of $n$-tuples this $g\in G$ is unique.

The following observations are easy, but important:

\begin{enumerate}
\item A group action of $G$ on $X$ is sharply $n$-transitive if
and only if for any $x\in X$, the \emph{stabilizer} $G_x=\{g\in G\colon x^g=x\}$ acts sharply $(n-1)$-transitively on $X\setminus\{x\}$.
\item If the action of $G$ on $X$ is sharply $n$-transitive, then
for any distinct elements $x_1,\ldots, x_n\in X$ the stabilizer
$G_{x_1,\ldots, x_n}$ is trivial.

\end{enumerate}

A sharply $1$-transitive action is also called \emph{regular}. It is easy to
see that,  after naming an element $x\in X$, a regular action induces a bijection
between the elements of $X$ and the elements of $G$: simply identify
$y\in X$ with the unqiue $g\in G$ such that $x^g=y$. After this identification
the action of $G$ on $X=G$ is nothing but the right regular action of the group
on itself by right multiplication. Conversely, right multiplication yields
a regular action of the group on itself. Thus we see that any regular action arises
in this way from right multiplication and there are no restrictions on the groups.
The situation changes drastically when one looks at higher degrees of transitivity.

Finite sharply $2$- and $3$-transitive groups were classified by Zassenhaus in \cite{Z1} and \cite{Z2} 
in the 1930's and were shown to arise
from so-called near-fields, as explained below. They essentially look like the groups
of affine linear transformations $x\mapsto ax+b$ or
Moebius transformations $x\mapsto \frac{ax+b}{cx+d}$, respectively.

For $n\geq 4$, the restrictions are even more severe:
Jordan proved in 1872 \cite{Jordan} that  apart from  
the symmetric and alternating group of degree $n$ and $n+2$, respectively,
the only finite sharply $n$-transitive groups are the Mathieu groups $M_{11}$ and $M_{12}$, which are sharply $n$-transitive for $n=4$ and $5$, respectively.

This was generalized by J. Tits and M. Hall, who proved that there are no infinite sharply $n$-transitive groups for $n\geq 4$.
However, it remained an open problem whether a classification similar to the one in the finite situation
holds for infinite sharply $2$- and $3$-transitive 
groups. Much literature
on this topic is available, see \cite{RST} for background and more recent references. 
In \cite{RST} the first construction of sharply $2$-transitive groups without nontrivial abelian normal subgroup is given. This was then extended in \cite{Tent3}
to sharply $3$-transitive actions.

These examples might inspire fresh interest in the classification problem and the aim of the article is to survey the background on sharply multiply transitive
permutation groups.

\section{A closer look at the classical examples}

Let us start with a closer look at the easiest examples:

{\bf The case $n=2$}.
Let $K$ be a (not necessarily commutative) field
and consider the group $G=\AGL(1,K)\cong K_+\rtimes K^*$ of affine
linear transformations  $x\mapsto ax+b, a,b\in K,a\neq 0$, which
acts sharply $2$-transitively on (the affine line of) $K$.
Then $G$ has a regularnormal subgroup $A\cong (K,+)$, i.e. a normal subgroup acting regularly on the underlying set,  consisting of the translations $x\mapsto x+b,b\in K$, and we see that the point stabilizer $G_0$
of $0\in K$ is the group of homotheties $x\mapsto ax,a\neq 0$, which
is obviously isomorphic to the multiplicative group $(K^*,\cdot)$ of the
underlying field. 

Since $A$ is normal in $G$, the elements of the point stabilizer $G_0$ act
by conjugation on the elements of $A$. The action of any $g\in G_0$ on $A$ is given by an automorphism of $A$. In other words, writing the group $A$ additively,
for all $n_1,n_2\in A, g\in G_0$ we have

\[(n_1+n_2)^g=n_1^g+n_2^g.\]

Thus, even without knowing that the group $G=\AGL(1,K)\cong A\rtimes G_0$ arises
from a field, the right distributivity appears automatically as a consequence just from the
group action.  For this distributivity, we do not even have to assume that the regular
normal subgroup in $G$ is abelian: if a sharply $2$-transitive group $G$ (acting on a set $X$)
contains a regular normal subgroup $N$, then for any $x\in X$ the point stabilizer 
$G_x$ acts by conjugation on the normal subgroup $N$ and since the sharp $2$-transitivity implies that $N\cap G_x=1$ 
 we necessarily have
\[ G\cong N\rtimes G_x.\]

Now for $g\in G_x, n_1,n_2\in N$ the action of $g$ on $N$ satisfies

\[(n_1\cdot n_2)^g=n_1^g\cdot n_2^g.\]

The point stabilizer $G_x$ thus acts as a group of automorphisms of $N$ which
is regular on the nontrivial elements of $N$. This already implies (among other things) that all elements of $N$ have the same order. 

What we have just seen is that any sharply $2$-transitive group $G$ which
has a regular normal subgroup gives rise to a \emph{near-field} in the sense of the following definition:

\begin{definition}
A (right) near-field is a structure $Q$, together with two binary operations, $+$ (addition) and $\cdot$ (multiplication), satisfying the following axioms:
\begin{enumerate}
\item   $(Q, +)$ is a group with identity element $0$ (we do
    not need to assume the group to be abelian);
\item $(Q\setminus\{0\},\cdot)$ is a group with identity element $1$;
\item $(a + b) \cdot c = a \cdot c + b \cdot c$ for all elements $a, b, c\in Q, c\neq 0$ (The right distributive law).
\end{enumerate}
 
\end{definition}

It is also easy to see that any near-field $Q$ gives rise to a sharply $2$-transitive group exactly in the same way in which $\AGL(1,K)$ arises from a 
field $K$ and that this group will again have a regular (and in fact abelian) normal subgroup
isomorphic to the additive group of $Q$. Namely for
$a,b\in Q, a\neq 0$, the group of transformations 
\[\phi_{a,b}\colon Q\longrightarrow Q, x\mapsto a\cdot x+b\]
acts sharply $2$-transitively on $Q$.
In this way, there is an one-to-one correspondence between near-fields and sharply $2$-transitive groups having a regular normal subgroup.

In the finite case, Zassenhaus proved that every sharply $2$-transitive group 
has a regular (and abelian) normal subgroup, but the question remained open whether this
remains true without the assumption of finiteness.

 It is interesting to note that while there are no
finite (proper) skew fields, finite near-fields do exist:

\noindent
{\bf Example:}    Let $K$ be the finite field of order $9$ and denote the field multiplication on $K$ by $*$. Using the action of the Frobenius automorphism $Frob_3: K\longrightarrow K \colon x\mapsto x^3$ we define a new binary operation `$\cdot$' on $K$ by the following rule:
For $a\in K$ we put  

$a \cdot b = a*b$     if $b\in K$ is a square, and
    
 $a\cdot b= Frob_3 (a)*b=a^3*b$ if $b\in K$ is not a square.

Then $K$ is a near-field with this new multiplication and the same addition as before.
    Near-fields arising in this way from a commutative field through twisting the multiplication by an automorphism of the field are called \emph{Dickson} near-fields.
    
In fact, Zassenhaus classified all finite near-fields by characterizing those finite (linear) groups that can act regularly on the set of non-zero vectors of a finite vector space.  He proved:

\begin{theorem}
Any finite sharply $2$-transitive group has a nontrivial abelian normal subgroup and thus
arises from a near-field.
All 
but seven finite near-fields are Dickson near-fields.
\end{theorem}
   
This gives rise to the following

\bigskip\noindent
{\bf Question:} Does every sharply $2$-transitive group contain a regular normal subgroup? 

Before answering this question we introduce the characteristic of a sharply $2$-transitive group:

\section{The characteristic of a sharply $2$-transitive group}

Any sharply $2$-transitive group $G$ acting on a set $X$
contains plenty of involutions: for any $x,y\in X, x\neq y$, there is a unique
element $t$ swapping $x$ and $y$. Since swapping twice fixes $x$ and $y$, we
see from part 2. of the observation above that $t^2=1$ and hence $t$ is an involution.

Given two involutions $t_1,t_2$ consider pairs $x_1^{t_1}=y_1, x_2^{t_2}=y_2$ with $x_i\neq y_i, i=1,2$. By $2$-transitivity there is some $g\in G$ such that $x_1^g=x_2,y_1^g=y_2$. Thus the conjugate $t_1^g=g^{-1}t_1g$ of $t_1$ under $g$ will swap $x_2$ and $y_2$. Since $t_2$ is the only element from $G$ swapping $x_2$ and $y_2$ by
$G$ being \emph{sharply} $2$-transitive, we see that $t_1^g=t_2$.

This shows that all involutions in $G$ are conjugate, so  the set
\[ J=\{t\in G\setminus\{1\}\colon t^2=1\}\]
 of involutions in $G$ forms a single conjugacy class.

In fact, if $t_1,t_2$ both do not fix $x_1$, we could have chosen $x_1=x_2$ and concluded that all involutions that do not fix $x_1$ are conjugate under $G_{x_1}$.
 
In particular, it follows that either all involutions of $G$ have a fixed point (and if they do each involution has a unique fixed point) or no involution has a
fixed point.

Let us first consider the case that involutions in $G$ have fixed points.
\begin{lemma} If $G$ is sharply $2$-transitive on $X$ and involutions in $G$ have fixed points, there is an obvious bijection
\[J\longrightarrow X, t\mapsto Fix(t)\]
where $Fix(t)$ denotes the fixed point of $t$.
\end{lemma}
\begin{proof}
 Since the elements are conjugate, this map is certainly surjective.
It is also injective: suppose $t_1,t_2\in J$ have the same fixed point $x$, say.
Let $z\in X$ be different from $x$ and let $z_1^{t_1}=z=z_2^{t_2}$.
Now by $2$-transitivity there is some $h\in G_z$ fixing $z$ and with $z_1^h=z_2$.
Then as above considerations we have $t_1^h=t_2$.
From 
\[x^{h^{-1}t_1h}=x^{t_2}=x\]
we obtain
\[x^{h^{-1}t_1}=(x^{h^{-1}})^{t_1}=x^{h^{-1}}\]

showing that $t_1$ fixes $x^{h^{-1}}$. Since $t_1$ has $x$ as its
unique fixed point, this implies that
that $h^{-1}$ and hence $h$ must fix $x$. But then $h$ fixes two distinct points, $x,z\in X$, and so
$h=1$ by sharp $2$-transitivity.
\end{proof} 
Obviously, the map from $J$ to $X$ given above is $G$-equivariant where
the action of $G$ on $J$ is given by conjugation.
Thus in the case where involutions have fixed points we have the following
consequences:
\begin{corollary}
If $G$ is sharply $2$-transitive on $X$ and involutions in $G$ have fixed points,
then
\begin{enumerate}
\item every point stabilizer $G_x$ contains a unique involution $t$, and this
involution is central in $G_x$;
\item the set
of products of two distinct involutions
forms a conjugacy class.
\end{enumerate}
\end{corollary}

The second part follows since an element $g\in G$ taking the pair  $(x_1,y_1)$ of distinct elements of $X$ to the pair $(x_2,y_2)$ will conjugate the (ordered) pair of involutions with fixed points $x_1,y_1$ to the pair with fixed point $x_2,y_2$.

If involutions have no fixed points, then not only does the previous argument 
not work, but in fact we will see below that $J^2\setminus\{1\}$ does not necessarily
form a single conjugacy class.

In the group $\AGL(1,K)$ the involutions are exactly the elements of the form $x\mapsto -x+b$. Thus the set 
\[J^2=\{t_1t_2\colon t_1,t_2\in J\}\]
equals the set of translations $x\mapsto x+b$, which is isomorphic to the additive group of the field $K$. Note that if the field $K$ is of characteristic $2$, then
the involutions $x\mapsto -x+b$ are fixed point free (and commute pairwise).
This explains and motivates the following definition
\begin{definition}
Let $G$ be a  sharply $2$-transitive group.
\begin{enumerate}
\item If involutions in $G$ have fixed points,
 we say that  $G$ has characteristic~$p$ if the elements of
$J^2\setminus\{1\}$ have  order $p$ and characteristic  $0$ if their order is
infinite.

\item If involutions in $G$ are fixed point free, we  say that $G$ has characteristic~$2$.
\end{enumerate}
If $G$ is sharply $3$-transitive on $X$, then the characteristic of $G$ is that of a point stabilizer $G_x$, $x\in X$.
\end{definition}

Note that the characteristic of a sharply $2$-transitive group is necessarily a prime $p$ 
and the case $p=2$ occurs if and only if involutions are fixed point free (as otherwise two distinct involutions cannot commute). Note however, that if the
characteristic of $G$ is $2$, then the involutions in $G$ commute if and only
if $J^2$ is a (commutative) subgroup.

A result of B. Neumann \cite{Neumann} now gives a partial answer to the above question or
at least a clear criterion:
\begin{theorem}
If $G$ acts sharply $2$-transitively on the set $X$, then $G$ contains a regular normal subgroup $N$ if and only if for some (or any) involution $t\in G$ the set 
\[tJ:=\{tt_1\colon t_1\in J\}\]
is a subgroup. In this case we
have $N=tJ=J^2$ and the group is abelian. 
\end{theorem}
A group in which $tJ$ is a subgroup is called a \emph{split} sharply $2$-transitive
group. The split sharply $2$-transitive groups are special cases of \emph{Moufang sets}
and there is extensive literature on these as well, see e.g. \cite{dMW}.

Note that Neumann's result says in particular that addition in nearfields is always commutative. Thus, any regular normal subgroup of a sharply $2$-transitive group is
abelian. Conversely, any non-trrivial abelian normal subgroup of a $2$-transitive group is regular.
Therefore we may just ask whether any sharply $2$-transitive group contains a non-trivial abelian normal subgroup.

It is  easy to see that the set $tJ$ -- whether it forms a subgroup or not -- acts regularly on the set $X$,
where in slight abuse of terminology we call the action of a subset  of $G$ regular on $X$ if for any $x,y\in X$ there is a unique element in the subset
taking $x$ to $y$.

Together with the (multiplicative) action of a point stabilizer $G_x$ on the
set $tJ$ we obtain the concept of a near domain, weakening the assumptions
of a near-field:
\begin{definition}A structure $(D,+,\cdot)$ is a \emph{near-domain} if $(D^*,\cdot)$ is a group,
$(D,+)$ is a loop (so not necessarily associative) and the structure is
right-distributive.

\end{definition}

Note that a near-domain $(D,+,\cdot)$ is a \emph{near-field} if and only if $(D,+)$ is a group.

What we have just seen is that to \emph{any} sharply $2$-transitive group  $G$ -- whether split or non-split -- one can associate  a near domain $D$, so that $G$ arises from $D$ exactly as $\AGL(1,K)$ arises from the field $K$. Conversely, any near domain gives rise to a sharply $2$-transitive group.

The near-domain $D$ associated to a sharply $2$-transitive group $G$ is a near-field if and only if $G$ contains a regular normal subgroup.

So we can reformulate the above question and ask instead:

\medskip
\noindent
{\bf Question:} Is any near-domain a near-field?
\medskip

By Zassenhaus's work the answer is 'yes' for finite near-domains.
However, for  infinite sharply $2$-transitive groups, it was a long standing open problem 
whether or not they all split.
Here are some splitting results in certain special cases:
\begin{itemize}
\item
In \cite{Ti} J.~Tits proved that if $G$ is locally compact  connected, then $G$ splits.

\item
In \cite{W} it was shown that if $G$ is locally finite, then $G$ splits.

\item
In \cite{T2} it was shown that if $G$ is  definable in an o-minimal structure, then $G$ splits.

\item
In \cite{GlGu} it was shown that if $G$ is linear (with certain additional restrictions) then $G$ splits.

\item
In \cite{GMS} it was shown that if $G$ is locally linear (with some additional restrictions) then $G$ splits.

\item
In \cite[Thm.~9.5, p.~42]{Kerby}  it was shown that if the characteristic of $G$ is three, then $G$ splits (see also \cite{Tu}).

\item
In \cite{M} it was shown that if the exponent of the point stabilizer is $3$ or $6,$ then $G$ splits. 
\end{itemize}

(See also \cite{BN} for further splitting results.)
However, it turns out that the answer to the guiding question is 'no'!
Not every sharply $2$-transitive group has a non-trivial abelian normal subgroup!

\section{New examples in characteristic $2$}

It turns out that not only is the answer to the above question 'No', but in fact it is
very far from being even 'close to yes':
it is  shown in \cite{RST} that in the infinite case, \emph{any} group can be embedded into
a group that acts sharply $2$-transitively on an appropriate set. 

Suppose we have a group $G$ acting on a set $X$. How can we extend the
group action (extending both the group and the set if necessary)
in such a way that the given action is preserved and the extended action becomes sharply $2$-transitive?

Obviously, in order for this to  work we must necessarily have to assume that
two-point stabilizers are trivial as otherwise this resulting group
could not be sharply $2$-transitive. 

\begin{definition}
A group $G$ acting on a set $X$ is called a Frobenius group (action) if
the action is transitive, but not regular and no nontrivial element fixes two 
elements from $X$.
\end{definition}

This turns out to be the \emph{only} restriction!
In \cite{RST} we show

\begin{theorem}
Every Frobenius or regular permutation group which is \emph{not}
sharply $2$-transitive, and  whose involutions
act without fixed points has a non-split sharply $2$-transitive extension of characteristic $2$.
\end{theorem}

\noindent
Here, by an ``extension'' we mean an extension of both the given group and the
underlying set in such a way that the original group action on the original
set is preserved. As pointed out in the beginning,  any group acts regularly
on itself by right multiplication. So \emph{any} group is a subgroup of a sharply $2$-transitive group.

The idea of the extension is rather simple: we need to make the action $2$-transitive and we have to make sure that the Frobenius condition is satisfied.
In the case of a free completion of a projective plane  one
adds new lines joining two given points and new points as interesections for two given lines until in the limit we obtain the required projective plane.

Incidentally, the examples constructed in \cite{RST} show that in characteristic $2$
the set $J^2\setminus\{1\}$ of products of two distinct involutions does not necessarily form a single conjugacy class.

As a special case we may start with the cyclic group $C_2$ of order $2$ and obtain:

\begin{theorem}(see \cite{TZ} for the general case)
The group $G=(C_2\times F_\omega)*F_\omega$ acts sharply $2$-transitively in characteristic $2$ (on some appropriate set), where $F_\omega$ denotes the free group of countably infinite rank.
\end{theorem}

Here, the $*$ denotes the free product of two groups: any element in the free product can
be written as a word whose letters alternate between elements from the different factors
of the product.
 The construction uses \emph{partial actions}: we start with the cyclic group $G_0=C_2=\langle t\rangle$ generated by the involution $t$ and assume that it acts on a four element set $X_0$. We then extend both the group $G_0$ and the underlying set $X_0$ in steps. At any given point, the involution $t$ acts everywhere on the given set $X_0$, and we assume that there are
always infinitely many generators for the free groups that do not yet act anywhere on the set.
Let us fix infinite generating sets $R$ and $S$ for the free groups to write 
\[G=(C_2\times F(R))*F(S).\]

We here do much the same as in the case of a free completion of a projective plane: we fix one pair of distinct points $(x,y)$ of the original set. We may assume  that they are swapped by the involution $t$. In order to make the group action sharply $2$-transitive
we have to make sure that for any pair $(w,z)$ of distinct elements from $X$
there is a group element - and a unique one - taking $(x,y)$ to $(w,z)$.

If such a group element already exists, there is nothing to do. Otherwise
we distinguish two cases: if the pair $(w,z)$ is not yet swapped by an 
involution, we add a new element $r$ to the group in form of a generator for 
an extension of the original group by a new free factor  and we extend the group action by letting a previously unused $r\in R$ take the pair $(x,y)$ to the pair $(w,z)$.

If on the other hand the pair $(w,z)$ is already swapped by the 
involution $t$, we add a new element $s\in S$ to the group in form of a generator for 
an extension of the original group  and we ask that $s$ commutes with $t$.  We then again extend the group action by letting $s$ take the pair $(x,y)$ to the pair $(w,z)$.

We then also have to extend the underlying set $X_0$ in order to let the 
new generators from $R$ and $S$ be defined everywhere.
This is essentially all. The details can be found in \cite{RST} in a group
theoretic language and in \cite{TZ} in a more direct approach using partial group actions. 

Using the criterion by B. Neumann, it is  not hard to show that the resulting 
group does not contain an abelian normal subgroup. Thus,
the question raised above is answered negatively, at least in 
characteristic $2$. 
However, in characteristic $2$ we do not have an easy description of the set that the sharply
$2$-transitive group acts on. Nevertheless we may ask whether these sharply $2$-transitive groups can 
occur as the point stabilizer of some sharply $3$-transitive group.

\section{Sharply $3$-transitive groups}

As in the case of sharply $2$-transitive groups, we first take
a closer look at the standard example of a sharply $3$-transitive group:

\bigskip
\noindent
{\bf The case $n=3$:} For any commutative field $K$ the group
$G=\PGL(2,K)$ acts sharply
 $3$-transitively on the projective line of $K$ via $x\mapsto \frac{ax +b}{cx+d}$.
 
Note that the point stabilizer $G_\infty$ in $G$ is (isomorphic to)
the group $\AGL(1,K)$.
So as before we can ask what we need of  the group $\PGL(2,K)$ to make a
sharply $3$-transitive group. As point stabilizers of sharply $3$-transitive
groups  are sharply $2$-transitive, we know from the previous section that
there is a near-domain associated to any sharply $2$-transitive group and
hence -- via point stabilizers - also to any $3$-transitive group.
By \cite{Kerby} Section 11, a near-domain $D$ gives rise to
a sharply $3$-transitive group if and only if there is an
 an involutory automorphism $\sigma$ of the multiplicative group $(D^+,\cdot)$ of $D$
which satisfies the functional equation

\[\sigma(1 + \sigma(x)) = 1-\sigma(1 + x) \mbox{\ for all \ } x \in D\setminus\{0,1\}.\]

 Unfortunately, it is not easy to describe the near-domains associated
to the construction of sharply $2$-transitive groups in \cite{RST,TZ}.
However, it turns out that
we can construct sharply $3$-transitive groups directly: 

Suppose that we have a group $G$ acting on a set $X$ in such a way that $3$-point stabilizers are trivial and that involutions fix a unique point. Then we can extend $G$ and $X$ to a sharply $3$-transitive action. As before we pick a
triple of distinct elements $(a,b,c)$ and we want to \emph{join} this triple to any other triple $(x,y,z)$ of distinct elements by a unique group element.
For simplicity we assume that the setwise stabilizer of $\{a,b,c\}$ in the group $G$ acts on the triple $(a,b,c)$ as $\Sym(3)=S_3$,
the symmetric group on three letters. We also assume that for any triple $(x,y,z)$ whose setwise stabilizer is isomorphic to $S_3$ there already exists a (unique)  $g\in G$ with $(a,b,c)^g=(x,y,z)$.

In order to construct the extension of $G$ and $X$, we here have to distinguish several cases: if the  triple $(x,y,z)$ is not
invariant under any group element, we extend the group $G$ by a free generator
taking $(a,b,c)$ to $(x,y,z)$. If $(x,y,z)$ is invariant under a single involution $\iota$, we extend $G$ by an HNN-extension conjugating one of the involutions of $G_{(a,b,c)}$ to $\iota$. If $(x,y,z)$ is invariant under 
an element $\tau$ of order three, we take an HNN-extension conjugating an
element of $G_{(a,b,c)}$ of order three  to $\tau$.
In the limit, the group we obtain will be sharply $3$-transitive of characteristic $2$ on some appropriate set (see \cite{Tent3} for details).

It is tempting to try to construct sharply $4$-transitive groups with a similar method. So one could fix a quadrupel of
points, assume that the stabilizer of this quadrupel is $S_4$ and try to extend the group. However -- luckily! -- this method fails:  if
we consider  two other quadrupels of points which are not yet joined to the fixed one, each invariant under some involution, then we need to join them by HNN-extensions that respect both relations and this cannot be done
in a free way.

\section{Other characteristics and further directions}

So far, the constructions described all lived in characteristic $2$. In other characteristics
we have to take the fixed points of involutions into account. In the construction of
sharply $2$-transitive groups described above if involutions have fixed points, then we cannot just extend the group action freely. In characteristic $0$,  recent progress shows \cite{RT} that this works. Since in this case, the action of the group $G$ on the underlying set is equivalent to the action by conjugation on the involutions of $G$, the situation is more canonical.

On the other hand, the problem is that if $G$ acts sharply $2$-transitively on $X$, then by the equivariant bijection between the set of involutions of $G$ and the underlying set $X$ 
we see that $G$ acts transitively on the set of products of two distinct involutions.

If the characteristic is different from $2$, then in particular all pairs of distinct involutions generate isomorphic dihedral groups. Adding free factors will necessarily
force the dihedral groups to be infinite. This allows us to construct sharply $2$-transitive groups of characteristic $0$ with a similar approach as in the characteristic $2$ case.

However, in characteristic $p>2$, the situation is much more difficult: not only are there no non-split sharply $2$-transitive groups in characteristic $p=3$ by Kerby's result mentioned above. But in the open cases with $p>3$, if we start with a group in which any pair of 
distinct involutions generates a dihedral group $D_{2p}$ of order $2p$ for some prime $p>3$, then extending the group in such a way that this property is preserved
demands great care as it bears similarity to the restricted Burnside problem.

The examples constructed so far suffice to conclude that the classes of sharply $2$- and $3$-transitive groups, respectively, are
wild as they include \emph{free} constructions. As far as classifications are concerned, this forces us to look in more
detail at more specific settings, as the list of splitting results shows that there are many natural classes in which sharply $2$-transitive groups split. One of the most urgent maybe the question whether there are non-split sharply $2$-transitive groups of finite Morley rank. This question might even have direct connections with
the existence of so-called 'bad' groups, i.e. simple groups of finite Morley rank which are not isomorphic to algebraic groups over algebraically closed fields.

We hope that the recent constructions of these examples inspire new interest and a fresh look at these natural questions.

%\begin{small}

%Fuer arxiv.org mu? man anders herum auskommentieren und urydirekt.bbl
% mitschicken.

%\bibliography{mrabbrev,urydirekt}

\begin{thebibliography}{99}
%\centerline{\textsc{Bibliography}
%%%%%%%%%%%%%%%%%%%%%%%%%%%%%%%%%%%%%%%%%
%%%%%%%%%%%%%%%%%%%%%%%%%%%%%%%%%%%%%%%%%
%%%%%%%%%%%%%%%%%%%%%%%%%%%%%%%%%%%%%%%%%

 

\bibitem[BN]{BN} A.~Borovik, A.~Nesin, {\it Groups of finite Morley rank,} Oxford Logic Guides, 26. 
Oxford Science Publications. The Clarendon Press, Oxford University Press, New York, 1994.

\bibitem[dMW]{dMW} T.~de Medts, R.~Weiss, {\it
Moufang sets and Jordan division algebras}, Math. Annalen 335 (2006), 415-433. 
 
\bibitem[GMS]{GMS} G.~Glauberman, A.~Mann, Y.~Segev, {\it  A note on groups generated by involutions
and sharply $2$-transitive groups,} to appear in Proc.~Amer.~Math.~Soc.

\bibitem[GlGu]{GlGu} Y.~Glasner, D.~Gulko, {\it Sharply~two~transitive~linear~groups,} to appear in
Int.~Math.~Res.~Not.

\bibitem[Ca]{Cameron}
Peter J. Cameron, {\it Bases in permutation groups}
in: Automorphisms of first-order structures.
Edited by Richard Kaye and Dugald Macpherson. Oxford Science Publications.
The Clarendon Press, Oxford University Press, New York, 1994. xiv+386 pp


\bibitem[DM]{DM}
J. D. Dixon, B.  Mortimer,
{\it Permutation groups},
Graduate Texts in Mathematics, 163. Springer-Verlag, New York, 1996. xii+346 pp. 


\bibitem[Ha]{hall}
M. Hall,
{\it On a theorem of Jordan}, Pacific J. of Math. 4 (1954) 219 -- 226.


\bibitem[Jo]{Jordan}
C. Jordan, Recherches sur les substitutions, J. Math. Pures Appl. (2) 17 (1872)
351-363. 

\bibitem[Ke]{Kerby}
W. Kerby, 
{\it On infinite sharply multiply transitive groups},
Hamburger Mathematische Einzelschriften, Neue Folge, Heft 6.
Vandenhoeck \& Ruprecht, G\"ottingen, 1974. 71 pp.
 
 

\bibitem[MaKS]{MaKS} W.~Magnus, A.~Karrass, and D.~Solitar, {\it Combinatorial group theory,}
Interscience, 1966.

\bibitem[M]{M} P.~Mayr, {\it Sharply $2$-transitive groups with point stabilizer of exponent $3$ or $6$,} 
Proc.~Amer.~Math.~Soc.~{\bf 134}  (2006),  no.~1, 9--13. 

\bibitem[MK]{MK} {\it The Kourovka notebook. Unsolved problems in group theory 18,} Editors: V.~D.~Mazurov and E.~I.Khukhro,
arXiv:1401.0300 (2014) 


\bibitem[Ne]{Neumann}
B. H. Neumann,
{\it On the commutativity of addition},
J. London Math Soc. 15 (1940), 203--208. 


\bibitem[RT]{RT} E. Rips,  K. Tent,
{\it Sharply $2$-transitive group in characterstic $0$}, preprint.

\bibitem[RST]{RST} E. Rips, Y. Segev, K. Tent,
{\it A sharply $2$-transitive group without a non-trivial abelian normal subgroup}, to
appear in JEMS.

\bibitem[TZ]{TZ} K. Tent, M. Ziegler,
{\it Sharply $2$-transitive groups}, Adv. Geom. 16 (2016), no. 1, 131 - 134. 


\bibitem[Te1]{T2} K.~Tent, {\it Sharply $n$-transitive groups in o-minimal structures,} Forum Math.~{\bf 12}  (2000),  no.~1, 65 - 75. 



\bibitem[Te2]{Tent3} K.~Tent, {\it Sharply $3$-transitive groups},
 Advances in Mathematics 286 (2016) 722 - 728. 


\bibitem[Ti]{Tits}
J. Tits, Groupes triplement transitifs et generalisations. Algebre et
Theorie de nombres, Coll. Int. du Centre Nat. de la Rech. Sci. no. 24 (1950) 207 -- 208.

\bibitem[Ti]{Ti} J.~Tits {\it Sur les groupes doublement transitifs continus,} Comm.~Math.~Helv.~26 (1952),
203--224.

\bibitem[Tu]{Tu} S.~T{\"u}rkelli, {\it Splitting of sharply 2-transitive groups of characteristic 3,}
Turkish J.~Math.~{\bf 28}  (2004),  no. 3, 295--298.

\bibitem[W]{W}  H.~W{\"a}hling, {\it Lokal endliche, scharf zweifach transitive Permutationsgruppen,} (German)
[Locally finite, sharply doubly transitive permutation groups] Abh.~Math.~Sem.~Univ.~Hamburg  {\bf 56}  (1986), 107--113.
\bibitem[Z1]{Z1} H. Zassenhaus,
{\it \"Uber endliche Fastk\"orper},
Abh. Math. Sem. Hamburg, {\bf 11} (1936), 187 -- 220.

\bibitem[Z2]{Z2} H. Zassenhaus,
{\it Kennzeichnung endlicher linearer Gruppen als Permutationsgruppen},
Abh. Math. Sem. Hamburg, {\bf 11} (1936), 17 -- 40.



\end{thebibliography}
%\bibliographystyle{plain}
%\input{urydirekt.bbl}

\vspace{.5cm}
%\end{small}

%\begin{small}
\noindent\parbox[t]{15em}{
Katrin Tent,\\
Mathematisches Institut,\\
Universit\"at M\"unster,\\
Einsteinstrasse 62,\\
D-48149 M\"unster,\\
Germany,\\
{\tt tent@wwu.de}}

%\end{small}
\end{document}